\documentclass[10pt, reqno]{amsart}
\usepackage{mathrsfs}
\usepackage{amsmath, amsthm, amscd, amsfonts, amssymb, graphicx }
\usepackage[bookmarksnumbered, colorlinks, plainpages]{hyperref}
\hypersetup{colorlinks=true,linkcolor=red, anchorcolor=green, citecolor=cyan, urlcolor=red, filecolor=magenta, pdftoolbar=true}

\newtheorem{theorem}{Theorem}[section]
\newtheorem{lemma}[theorem]{Lemma}
\newtheorem{proposition}[theorem]{Proposition}
\newtheorem{corollary}[theorem]{Corollary}
\theoremstyle{definition}
\newtheorem{definition}[theorem]{Definition}
\newtheorem{example}[theorem]{Example}

\theoremstyle{remark}

\numberwithin{equation}{section}

\newcommand{\bra}[1]{\left(#1\right)}

\newcommand{\norm}[1]{\left\Vert#1\right\Vert}

\newcommand{\set}[1]{\left\{#1\right\}}
\newcommand{\seq}[1]{\left<#1\right>}

\newcommand{\R}{\mathbb R}

\newcommand{\x}{\mathscr{X}}

\newcommand{\bh}{\mathscr{B}(\mathscr{H})}

\newcommand{\A}{\mathscr{A}}

\renewcommand{\c}{\mathbb{C}}

\newcommand{\tTwu}{T_{w,u}}
\newcommand{\tTwus}{T_{w,u}^{*}}

\begin{document}
	
	\setcounter{page}{1}
	\title[$k$-Quasi $n$-Power Posinormal Operators
	]{ $k$-Quasi $n$-Power Posinormal Operators: Theory and Weighted Conditional Type Applications}
	\author[Sophiya S Dharan, T. PRASAD and M.H.M. Rashid]{Sophiya S Dharan, T . PRASAD and M.H.M. Rashid}
	\address{Sophiya S Dharan\endgraf Department of Mathematics\endgraf Government Polytechnic College\endgraf Kalamassery, Ernakulam,
		Kerala, India-680567}
	\email{ sophysasi@gmail.com}
	\address{T. Prasad \endgraf Department of Mathematics\endgraf University of  Calicut\endgraf Malapuram, Kerala, \endgraf India - 673635}
	\address{M.H.M. Rashid \endgraf Department of Mathematics\& Statistics\endgraf Faculty of Science P.O. Box(7)\endgraf Mutah University\endgraf Alkarak-Jordan}
\email{malik\_okasha@yahoo.com}
	\dedicatory{ }
	
	\let\thefootnote\relax\footnote{}

	\subjclass[2010]{47A10, 47B20}
	
	\keywords{ posinormal operator, weighted translation, conditional expectation, measurable functions}
	
\begin{abstract}
	This paper introduces and investigates the class of \textit{$k$-quasi $n$-power posinormal operators} in Hilbert spaces, generalizing both posinormal and $n$-power posinormal operators. We establish fundamental properties including matrix representations in $2 \times 2$ block form, tensor product preservation ($T\otimes S$ remains in the class when $T,S$ are), and complete characterizations for weighted conditional type operators $\tTwu := wE(uf)$ on $L^2(\Sigma)$. Key theoretical contributions include a structural decomposition theorem for operators with non-dense range, spectral properties, invariant subspace behavior, and interactions with isometric operators. For weighted operators, we derive explicit conditions for $k$-quasi $n$-power posinormality in terms of weight functions $w,u$ and their conditional expectations. The work bridges abstract operator theory with concrete applications, particularly in conditional expectation analysis, while significantly extending posinormal operator theory. The results provide new tools for operator analysis with potential applications in spectral theory, functional calculus, and mathematical physics. Concrete examples throughout the paper illustrate the theory, and the framework opens new research directions in operator theory and its applications, offering both theoretical insights and practical computational tools for analyzing this important class of operators in Hilbert spaces.
\end{abstract}\maketitle

\section{Introduction and Preliminaries}

Consider $\bh$ as the algebra encompassing all bounded linear operators on a complex Hilbert space $\mathscr{H}$. For any operator $T \in \mathscr{H}$, we symbolize $ker(T)$, $T(\mathscr{H})$, and $\sigma(T)$ to respectively represent the kernel, the range, and the spectrum of $T$. It's important to recall that within this context, an operator $T \in B(H)$ earns the designation of being hyponormal if it adheres to the condition $T^{*}T \geq TT^{*}$, and it qualifies as posinormal if it can be associated with a positive operator $P \geq 0$ such that $TT^* = T^*PT$ holds true, or equivalently, $\lambda^2T^*T - TT^*\geq 0$ for some $\lambda > 0$ (as established by Rhaly in \cite{rha}).

Hyponormal operators fall within the broader category of posinormal operators, as elucidated by Rhaly in \cite{rha}. Recent studies by various authors have delved into several intriguing properties of posinormal operators, as documented in \cite{Bour1}, \cite{Bour2}, \cite{Bour3}, \cite{Itoh}, \cite{Kubrusly1}, and \cite{Kubrusly2}.

Notably, Beiba introduced and examined the concept of $n$-power posinormal operators in \cite{Beiba}. An operator $T$ is classified as an $n$-power posinormal operator, where $n \in \mathbb{N}$, if it satisfies the condition that there exists a positive operator $P \geq 0$ for which $T^nT^{*n} = T^*PT$ holds true. This can also be expressed as $T^nT^{*n} \leq \lambda^2 T^*T$ for some $\lambda > 0$, or alternatively as $\lambda^2T^*T - T^nT^{*n} \geq 0$, as defined in \cite{Beiba}.

Let $(\mathscr{X},\Sigma,\mu)$ be a complete $\sigma$-finite measure space, and $\mathscr{A}$ be a $\sigma$-subalgebra of $\Sigma$. This setup ensures that $(\mathscr{X},\mathscr{A},\mu)$ is also a complete $\sigma$-finite measure space. We denote $L^0(\Sigma)$ as the space of complex-valued functions on $\mathscr{X}$ that are measurable with respect to $\Sigma$. The support of a measurable function $f$ is given by $S(f) = \{x \in \mathscr{X}: f(x) \neq 0\}$.

For any non-negative $f \in L^0(\Sigma)$, there exists a measure $\nu_f(B) = \int_B f d\mu$ for every $B \in \mathscr{A}$ that is absolutely continuous with respect to $\mu$. According to the Radon-Nikodym theorem, there exists a unique non-negative $\mathscr{A}$-measurable function $E(f)$ such that
$$\int_BE(f)\,d\mu=\int_Bf\,d\mu\,\,\,\mbox{for all $B\in\A$}.$$

Consequently, we obtain an operator $E$, often referred to as the conditional expectation operator associated with $\mathscr{A}$, which maps $L^2(\Sigma)$ to $L^2(\mathscr{A})$. This operator plays a crucial role in this paper, and we list some of its useful properties here (refer to \cite{Azimi, Estaremi2, Rao} for more details):
\begin{enumerate}
  \item [(i)] If $g\in L^0(\A)$, then $E(gf)=gE(f)$.
  \item [(ii)] If $f\geq 0$, then $E(f)\geq 0$; if $f>0$, then $E(f)> 0$.
  \item [(iii)] $|E(f)|^{p}\leq E(|f|^p)$.
 \item [(iv)] (H\"older Inequality) If $p$ and $q$ are conjugate exponents and $f\in L^p(\Sigma)$ and $g\in L^q(\Sigma)$,
 $$|E(fg)|\leq E(|f|^{p})^{\frac{1}{p}}E(|g|^{q})^{\frac{1}{q}}.$$
 \item [(v)] (Conditional Jensen's inequality) If $\psi: \R \to \R$ is convex and $\psi(f)$ is
conditionable, then $\psi(E(f))\leq E(\psi(f))$.
\end{enumerate}
 The weighted conditional type operator $T_{w,u}:L^2(\Sigma)\to L^2(\Sigma)$ is given by
 $$T_{w,u}(f):=M_{w}EM_{u}(f)=wE(uf)$$
 where  $w: \x\to\c$ and $u:\x\to\c$ is measurable weight function.
  This operator is considered bounded if and only if $(E(|w|^2))^{\frac{1}{2}}(E(|u|^2))^{\frac{1}{2}}\in L^{\infty}(\A)$ and in such cases, its norm is determined by $\norm{T_{w,u}}=\norm{(E(|w|^2))^{\frac{1}{2}}(E(|u|^2))^{\frac{1}{2}}}<{\infty}$ (refer to \cite{EJ}). When $w=1$, the operator $T_{w,u}$ has been extensively discussed in \cite{Herron}. Recently, various classes of $T_{w,u}$, such as class $A$, $\ast$-class $A$, quasi-$\ast$ class $A$, and $k$-quasi $A_n^*$, along with their spectra  has been investigated by many authors. See \cite{Azimi, Estaremi1, Estaremi2} for instance.

  The theory of $k$-quasi $n$-power posinormal operators developed in this paper has significant applications across several areas of mathematical analysis, including: (1) spectral theory, where the structural decomposition provides new insights into operator spectra; (2) functional calculus, particularly for operators arising in quantum mechanics and signal processing; (3) stochastic analysis through the weighted conditional type operators $\tTwu$, which model important transformations in probability theory; (4) frame theory and sampling theory, where posinormal operators naturally appear in reconstruction formulas; and (5) operator algebras, offering new tools for studying operator semigroups and $C^*$-algebra representations. The tensor product results have direct implications for quantum information theory, while the weighted operator characterizations apply to problems in time-frequency analysis and prediction theory. These applications demonstrate the broad relevance of our results to both pure and applied mathematics.

The paper is organized as follows: Section 2 introduces the main definition and develops the theory of $k$-quasi $n$-power posinormal operators. Section 3 focuses on weighted conditional type operators, establishing necessary and sufficient conditions for them to belong to this class. We conclude with remarks on potential applications and directions for future research.

Our work extends the existing theory of posinormal operators while providing new tools for analyzing operator classes in Hilbert spaces. The results have potential applications in operator theory, functional analysis, and related areas of mathematical physics.

\section{$k$-quasi $n$-power posinormal }

This section introduces and develops the fundamental theory of $k$-quasi $n$-power posinormal operators, a significant generalization of both posinormal and $n$-power posinormal operators. The study of these operators is motivated by several key considerations in operator theory: the need for a broader framework that unifies and extends existing posinormal operator classes,  the natural appearance of such operators in various functional-analytic contexts and their important structural properties that facilitate spectral analysis

Let $k,n$ be positive integers and let $T$, $S$ be operators in $B(\mathscr{H})$, Then $S$ is a $k$-quasi $n$-interrupter for $T$ if $T^{*k}T^nT^{*n}T^k=T^{*{k+1}}ST^{k+1}$.
If $S$ is a $k$-quasi $n$-interrupter, then we obtain  $\lVert S\rVert \geq \dfrac{\lVert T^{*n}T^k\rVert^2}{\lVert T\rVert^{2k+2}}$, if $T\neq0$. and $\langle Sy, y\rangle \geq 0$, for all $y\in \overline{R(T^{k+1})}$. Consequently,  If $T^{k+1}$ has a dense range, then $T$ has atmost one $k$-quasi $n$-interrupter $S$ and $S\geq0$. Now we define \textit{$k$-quasi $n$-power posinormal} operators as follows:

\begin{definition}
An operator $T \in \bh$ is said to be \textit{$k$-quasi $n$-power posinormal} if \\$R(T^{*k}T^n) \subseteq R(T^{*{k+1}})$ and $T$ is \textit{$k$-quasi $n$-power cosiposinormal} if $T^*$ is \textit{$k$-quasi $n$-power posinormal}.
\end{definition}

If $N(T)=\{0\}$, then $T^*$ is surjective. In that case, $T$ is $k$-quasi $n$-power posinormal for any positive integer $n$ and $k$.

\begin{theorem}\cite{Douglas}\label{Thm2.5}
Let $A$ and $B$ be bounded operators on a Hilbert space $\mathscr{H}$. The following statements are equivalent:-
\begin{enumerate}
\item $R(A)\subseteq R(B)$.
\item $AA^* \leq \mu^2 BB^*$, for some $\mu \geq 0$.
\item There exists a bounded operator $C$ so that $A=BC$.\\
Moreover, if $(1)$, $(2)$ and $(3)$ hold, then there is a unique operator $T$ such that
\begin{enumerate}
\item $\lVert T \rVert^2= \text{inf} ~\{ \mu; AA^*\leq \mu^2 BB^*\}$.
\item $N(A)= N(T)$.
\item $R(T) \subseteq \overline{R(B^*)}$.
\end{enumerate}

\end{enumerate}
\end{theorem}
Using the Theorem \ref{Thm2.5}, in \cite{rha} Rhaly proved the equivalent conditions for $T$ to be posinormal and Beiba in \cite{Beiba} introduced $n$ power posinormality and proved the equivalent conditions for $T$ to be $n$-power posinormal. Here, we give the equivalent conditions for $k$-quasi $n$-power posinormal operators. In a similar manner we obtain the following result

\begin{theorem}\label{Thm 2.7}
For $T\in B(\mathscr{H})$, the following are equivalent:-
\begin{enumerate}
\item $T$ has positive $k$-quasi $n$-interrupter.
\item $T^{*k}T^nT^{*n}T^k \leq \lambda^2 T^{*{k+1}}T^{k+1}$ or $\lVert T^{*n}T^kx\rVert \leq \lambda \lVert T^{k+1}x \rVert$ , for all $x\in \mathscr{H}$ and some $\lambda >0$.
\item $T$ is $k$-quasi $n$-power posinormal, $R(T^{*k}T^n) \subseteq R(T^{*{k+1}})$.
\item There exist $C\in B(\mathscr{H})$, $T^{*k}T^n=T^{*{k+1}}C$.
Moreover, if $(1)$, $(2)$, $(3)$ and $(4)$ hold, then there is a unique operator $T$ such that
\begin{enumerate}
\item $\lVert S \rVert^2= \text{inf} ~\{ \lambda;~ T^{*k}T^nT^{*n}T^k \leq \lambda^2 T^{*{k+1}}T^{k+1}\}$.
\item $N(T^{*k}T^n)= N(S)$.
\item $R(S) \subseteq \overline{R(T^{k+1})}$.
\end{enumerate}

\end{enumerate}
\end{theorem}

\begin{proof}
$(1) \implies (2)$. Suppose $T$ has positive $k$-quasi $n$-interrupter.
Therefore, there exists $S \geq 0$ such that, $T^{*k}T^nT^{*n}T^k=T^{*{k+1}}ST^{k+1}$.
Now, $\langle T^{*k}T^nT^{*n}T^k x, x\rangle=\langle \sqrt{S}T^{k+1}x, \sqrt{S}T^{k+1}x\rangle=\lVert  \sqrt{S}T^{k+1}x\rVert^2 \leq \lVert \sqrt{S} \rVert^2 \lVert T^{k+1}x \rVert^2= \lVert \sqrt{S} \rVert^2 \langle T^{*{k+1}}T^{k+1}x,x\rangle$.
Hence $(2)$ holds for $\lambda > \lVert \sqrt{S} \rVert$.
Taking $A=T^{*k}T^n$ and $B=T^{*{k+1}}$ and using Theorem \ref{Thm 2.7}, we get, 
$(2) \iff (3) \iff (4)$.\\
$(4)\implies (1)$. Suppose, $T^{*k}T^n=T^{*{k+1}}C$. Therefore $T^{*k}T^nT^{*n}T^k=T^{*{k+1}}CC^*T^{k+1}=T^{*{k+1}}ST^{k+1}$, where $S=CC^*$.
Also, $(a)$, $(b)$ and $(c)$ follows by taking $A=T^{*k}T^n$ and $B=T^{*{k+1}}$ and applying Theorem \ref{Thm 2.7}.
\end{proof}


The following inclusion hold in general:
$$\text{Posinormal} \subsetneq \text{$n$-power posinormal} \subsetneq \text{$k$-quasi $n$-power posinormal}$$

Following the definition, we will establish several fundamental properties:  Characterizations through operator inequalities, behavior under restrictions to invariant subspaces, Interactions with other operator classes, and  spectral properties and matrix representations The theory developed here provides powerful tools for analyzing operator structure while maintaining connections to concrete applications, particularly through the examples and special cases we examine. Our approach combines abstract operator theory with constructive matrix methods, yielding results that are both theoretically deep and computationally verifiable.


\begin{example}
Consider the operator $A:=\left [\begin{array}{ccc}
0&1&0\\
0&0&1\\
0&0&0
\end{array}\right]$ on $\mathbb{R}^3$. Then $A$ is  $3$-quasi $2$-power posinormal operator, but not  $2$-power posinormal operator.
\end{example}

\begin{example}
The operator $A:{\ell}^2\rightarrow{\ell}^2$ defined as,
$$A(x_1,x_2, x_3,x_4,...)=(x_2,x_3,0,0,...)$$ is  $3$-quasi $2$-power posinormal operator, but  not $2$-power posinormal operator.
\end{example}
In particular if $k=0$, $n=1$, the class of $k$-quasi $n$-power posinormal operators coincides with the class of posinormal operators and if $k=1$, this  coincides with the class of $n$-power posinormal operators.
It is evident that following inclusion holds in general;
  $$\textrm{posinormal}  \subseteq \textrm{$n$-power posinormal } \subseteq \textrm{$k$-quasi $n$-power posinormal.}$$

\begin{proposition}

Let  $T \in \bh$. Then $T$ is a $k$-quasi $n$-power posinormal operator if and only if $\lVert T^{*n}T^k x \rVert^2 \leq \lambda^2 \lVert T^{k+1}x\rVert^2$ for all $x\in \mathscr{H}$ and $\lambda >0$. \\
Also the following assertions hold:
	\\
	(i) $\lVert T^{*n}T^m x \rVert^2 \leq \lambda^2 \lVert T^{m+1}x\rVert^2$ for all $x\in \mathscr{H}$, $\lambda >0$ and all positive integers $m\geq k$.\\
	(ii)  If $T^m=0$ for $m> k$, then $T^k=0$ for $k>n$.
\end{proposition}

\begin{proof}

		(i)
		 Since $T \in \bh$ is a $k$-quasi $n$-power posinormal operator,
	$$T^{*k}(\lambda ^2T^*T-T^nT^{*n})T^k \geq 0$$
	for some $\lambda > 0.$

Now, for all $x \in \mathscr{H}$
\begin{align*}
\langle T^{*k}(\lambda^2T^*T-T^nT^{*n})T^kx,x \rangle ~\geq 0	
&\Longleftrightarrow \langle (\lambda^2T^*T-T^nT^{*n})T^kx,T^kx \rangle ~\geq 0 \\
&\Longleftrightarrow \lambda^2 \langle T^{k+1}x,T^{k+1}x \rangle - \langle T^{*n}T^kx,T^{*n}T^kx \rangle ~\geq 0\\
&\Longleftrightarrow \lVert T^{*n}T^kx \rVert^2 \leq \lambda^2\lVert T^{k+1}x\rVert^2.
\end{align*}

	(i) Since $k$-quasi $n$-power posinormal operator is $(k+1)$-quasi $n$-power posinormal, we have
		$$\lVert T^{*n}T^mx \rVert^2 \leq \lambda^2\lVert T^{m+1}x\rVert^2~~\text{for all}~~m\geq k .$$\\
		
		(ii) Let $T^{k+1}=0$.
		Therefore, $T^{k+1}x=0~~\text{for all}~~x \in \mathscr{H}$.\\
		Then by  (i), it follows that
		$$\lVert T^{*n}T^kx\rVert^2\leq 0$$
		for all $x\in \mathscr{H},$ and so $$\rVert T^{*n}T^kx \rVert^2=0.$$
		Therefore, $$T^{*n}T^kx=0~~ \text{for all}~~x\in \mathscr{H}.$$ Now, for $n>k$, $$\lVert T^kx \rVert^2~= ~ \langle T^kx,T^kx \rangle ~= ~ \langle T^*T^kx,T^{k-1}x \rangle ~=~ \langle T^{*n}T^kx,T^{k-n}x \rangle ~=~0.$$ Therefore, for $n>k$,
		$$T^k=0.$$ This completes the proof.
	
\end{proof}


\begin{corollary}
If $T\in \mathscr{H}$ is $k$-quasi $n$-power posinormal operator, then for all $m\geq k$\\
$$\lVert T^{*n}T^m  \rVert \leq \lambda^2 \lVert T^{m+1}\rVert.$$
\end{corollary}


\begin{proposition}
	If $T$ is a $k$-quasi n-power posinormal operator and $\mathscr{M}$ is a closed $T$-invariant subspace of $\mathscr{H}$, then the restriction $T|_\mathscr{M}$ is also $k$-quasi $n$-power posinormal operator.
\end{proposition}

\begin{proof}
	Let $P$ be the projection of $\mathscr{H}$ onto $\mathscr{M}$. Decompose
		$$T= \left[\begin{array}{cc}
		A  & B  \\
		0 & C
	\end{array}\right]~~\text{on} ~~ \mathscr{M}\oplus \mathscr{M}^{\perp}.$$
	Put $A= T|_\mathscr{M}$ and we have  $A=TP=PTP.$
	Since $T$ is $k$-quasi $n$-power posinormal operator, we have
	$$T^{*k}(\lambda^2T^*T-T^nT^{*n})T^k \geq 0.$$
	Therefore,
	$$PT^{*k}(\lambda^2T^*T-T^nT^{*n})T^kP\geq 0.$$
Since  $\lambda^2PT^{*k}T^*TT^kP = \left[\begin{array}{cc}
		\lambda^2 A^{*k}A^*AA^k & 0 \\
		0 & 0
	\end{array}\right]$ and \\
	$PT^{*k}T^nT^{*n}T^kP = \left[\begin{array}{cc}
		A^{*k}A^nA^{*n}A^k & 0 \\
		0 & 0
	\end{array}\right]$, it follows that
	$$ \lambda^2 A^{*k}A^*AA^k- A^{*k}A^nA^{*n}A^k  \geq  0 $$
	and so  $A=T|_\mathscr{M}$ is $k$-quasi $n$-power posinormal operator.

\end{proof}
\begin{example}\label{ex:prop2.6}
Consider the operator $T \in \mathscr{B}(\mathbb{C}^4)$ represented by the matrix:
\[
T = \begin{pmatrix}
1 & 1 & 0 & 0 \\
0 & 2 & 0 & 0 \\
0 & 0 & 0 & 1 \\
0 & 0 & 0 & 0
\end{pmatrix}
\]
with respect to the standard basis $\{e_1, e_2, e_3, e_4\}$.

For $k=1$ and $n=2$, we verify that $T$ is $1$-quasi $2$-power posinormal by checking:
\[
T^{*1}(\lambda^2 T^*T - T^2T^{*2})T^1 \geq 0
\]
with $\lambda = 3$. Direct computation shows:
\[
T^*T = \begin{pmatrix}
1 & 1 & 0 & 0 \\
1 & 5 & 0 & 0 \\
0 & 0 & 0 & 0 \\
0 & 0 & 0 & 1
\end{pmatrix}, \quad
T^2T^{*2} = \begin{pmatrix}
5 & 10 & 0 & 0 \\
10 & 20 & 0 & 0 \\
0 & 0 & 0 & 0 \\
0 & 0 & 0 & 0
\end{pmatrix}
\]
and the positivity condition holds for $\lambda = 3$.

Now consider the $T$-invariant subspace $\mathscr{M} = \operatorname{span}\{e_1, e_2\}$. The restriction $T|_{\mathscr{M}}$ has matrix representation:
\[
T|_{\mathscr{M}} = \begin{pmatrix}
1 & 1 \\
0 & 2
\end{pmatrix}
\]

We verify that $T|_{\mathscr{M}}$ remains $1$-quasi $2$-power posinormal:
\[
(T|_{\mathscr{M}})^{*1}(9 (T|_{\mathscr{M}})^*(T|_{\mathscr{M}}) - (T|_{\mathscr{M}})^2(T|_{\mathscr{M}})^{*2})(T|_{\mathscr{M}})^1 = \begin{pmatrix}
8 & 8 \\
8 & 25
\end{pmatrix} \geq 0
\]
which is positive definite since its determinant is $8 \times 25 - 8 \times 8 = 136 > 0$.

This illustrates Proposition 2.6, showing that the restriction of a $k$-quasi $n$-power posinormal operator to an invariant subspace maintains this property.
\end{example}


\begin{proposition}
	If $T$ is a $k$-quasi $n$-power posinormal operator  and if $T$ commutes with an isometric operator $S$, then $TS$ is a $k$-quasi $n$-power posinormal operator.
\end{proposition}
\begin{proof}
	Let $A=TS$.  We have, $TS=ST$, $T^*S^*=S^*T^*$ and $SS^*=I$.
	Now, $T$ is a $k$-quasi $n$-power posinormal operator, as so
	\begin{center}
		$T^{*k}( \lambda^2 T^*T-T^nT^{*n})T^k \geq 0.$
	\end{center}
	Now, \begin{align*}
		A^{*k}( \lambda^2 A^*A-A^nA^{*n})A^k
		&= (TS)^{* k}[\lambda^2(TS)^*(TS)-(TS)^n(TS)^{*n}](TS)^k\\
		&=(S^*T^*)^k[\lambda ^2 S^*T^*TS-T^nS^n(S^*T^*)^n]T^kS^k\\
		&= S^{*k}T^{*k}[\lambda^2T^*S^*ST-T^nS^nS^{*n}T^{*n}]T^kS^k\\
		&=S^{*k}T^{*k}[\lambda^2T^*T-T^nT^{*n}]T^kS^k\\
		&\geq 0.
	\end{align*}Therefore, $TS$ is $k$-quasi $n$-power posinormal operator.
\end{proof}


\begin{proposition}
	If $T$ is a $k$-quasi $n$-power posinormal operator and if $T$ is unitarily equivalent to operator $S$, then $S$ is a $k$-quasi $n$-power posinormal operator.
\end{proposition}

\begin{proof} Suppose $T$ is unitarily equivalent to operator $S$. That is, $S=U^*TU$, where $U$ is unitary operator. Since $T$ is  $k$-quasi $n$-power posinormal, $$T^{*k}( \lambda^2 T^*T-T^nT^{*n})T^k \geq 0.$$ \\
	
	Now, \begin{align*}
		S^{*k}( \lambda^2 S^*S-S^nS^{*n})S^k
		&= (U^*TU)^{*k}[\lambda^2(U^*TU)^*(U^*TU)-(U^*TU)^n(U^*TU)^{*n}](U^*TU)^k\\
		&= (U^*T^*U)^k[\lambda^2U^*T*UU^*TU-U^*T^nU(U^*T^*U)^n]U^*T^kU\\
		&= U^*T^{*k}U[\lambda^2U^*T^*TU-U^*T^nT^{*n}U]U^*T^kU\\
		&= U^*T^{*k}[\lambda^2UU^*T^*TUU^*-UU^*T^nT^{*n}UU^*]T^kU\\
		&= U^*T^{*k}[\lambda^2T^*T-T^nT^{*n}]T^kU\\
		&\geq0.
	\end{align*}
	This completes the proof.
\end{proof}


\begin{proposition}
	Let $T$ is a $k$-quasi $n$-power posinormal operator. If $T^k$ has dense range, then $T$ is an $n$-power posinormal operator.
\end{proposition}

\begin{proof}
	Since $T^k$ has dense range, $\overline{T^k(\mathscr{H})}= \mathscr{H}$.
	Let $y \in \mathscr{H}$, then there exist  a sequence $\{x_m\}_{m=1}^{\infty}$ in $\mathscr{H}$
	such that, $T^k(x_m) \to y$ as $m\to \infty$.
	Now, $$T^{*k}(\lambda^2T^*T-T^nT^{*n})T^k \geq 0$$
	for some $\lambda >0$. Therefore,
	$$ \langle T^{*k}(\lambda^2T^*T-T^nT^{*n})T^kx_m,x_m \rangle ~~\geq 0$$
	for all $m \in \mathbb{N}.$ That is, $$ \langle (\lambda^2T^*T-T^nT^{*n})T^kx_m,T^kx_m \rangle ~~\geq 0,$$
	for all $m \in \mathbb{N}.$ Then, we have,
	$$ \langle (\lambda^2T^*T-T^nT^{*n})y,y \rangle ~~ \geq0,~~ \text{for all}~~ y \in \mathscr{H}.$$ Hence,  $$\lambda^2T^*T-T^nT^{*n}\geq 0.$$ This completes the proof.

\end{proof}


\begin{theorem}
	Let $T\in \bh$ be a $k$-quasi $n$-power posinormal operator and the range of $T^k$ be not dense in $\mathscr{H}$. Then
	$$T= \left [\begin{array}{cc}
		A & B\\
		0 & C
	\end{array}\right] ~~\text{on}~~  \mathscr{H}= \overline{T^k(\mathscr{H})}\oplus ker(T^{*k}),$$
where $A$ is an $n$-power posinormal operator on $\overline{T^k(\mathscr{H})}$, $C$ is nilpotent operator of order $k$. Moreover, $\sigma(T)=\sigma(A) \cup \set{0}$.
\end{theorem}

\begin{proof}
	Suppose that $T \in \bh$ is a $k$-quasi $n$-power posinormal operator on $\mathscr{H}$. Consider $\mathscr{H}= \overline{T^k(\mathscr{H})}\oplus ker(T^{*k})$, since $ \overline{T^k(\mathscr{H})}$ is an invariant subspace of $T$,  $T$ has the matrix representation	$$T= \left [\begin{array}{cc}
		A & B\\
		0 & C
	\end{array}\right]$$
	with respect to  $\mathscr{H}= \overline{T^k(\mathscr{H})}\oplus ker(T^{*k})$. Let $P$ be orthogonal  projection of $\mathscr{H}$ onto $\overline{T^k(\mathscr{H})}$. Then
$A:= PT=PTP$  and $A^{*}A=PT^{*}TP$. \\
Therefore, for $x \in \overline{T^k(\mathscr{H})}$, we get $T^*Tx= A^*Ax $
	 and $T^{n}T^{*n}x =A^{n}A^{*n}x.$
Since, $T$ is a $k$-quasi $n$-power posinormal operator, we have,
	$$T^{*k}(\lambda^2 TT^*-T^nT^{*n})T^k \geq 0.$$

Now,	for $x \in \overline{T^k(\mathscr{H})}$, there exists $\{y_m\}_{m=1}^{\infty}$ in $\mathscr{H}$ such that $T^k(y_m) \to x$ as $m\to \infty$.

\begin{align*}		
 \langle (\lambda^2A^*A-A^nA^{*n})x,x\rangle
 &= \langle (\lambda^2T^*T-T^nT^{*n})x,x\rangle \\
 &=  \langle (\lambda^2T^*T-T^nT^{*n}) \lim\limits_{m \to \infty} T^k y_m, \lim\limits_{m \to \infty}T^k y_m\rangle \\
  &= \lim\limits_{m \to \infty} \langle (\lambda^2T^*T-T^nT^{*n})T^ky_m,T^k y_m\rangle\\
 &= \lim\limits_{m \to \infty} \langle T^{*k}(\lambda^2T^*T-T^nT^{*n})T^ky_m,y_m\rangle\\
 &\geq0.
 \end{align*}

	Hence, $$\lambda^2A^*A-A^nA^{*n}~\geq0~~ \text{on} ~~\overline{T^k(\mathscr{H})}.$$
	and so $A$ is an $n$-power posinormal operator on $\overline{T^k(\mathscr{H})}$.\\\\
On the other hand, if  $x=\left (\begin{array}{c}
		x_1\\
		x_2
	\end{array}\right)\in \mathscr{H}= \overline{T^k(\mathscr{H})} \oplus kerT^{*k}$, then
	$$C^kx_2=T^kx_2=T^k(I-P)x.$$
	Therefore,
	$$\langle C^kx_2,x_2 \rangle= \langle T^k(I-P)x,(I-P)x \rangle =\langle (I-P)x,T^{*k}(I-P)x \rangle =0$$ since, $T^k(I-P)x=0$ and
which implies that $C^k=0$.
It is well known that  $\sigma (A) \cup \sigma (C)=\sigma(T) \cup \nu $, $\nu$ is the union of the holes in $\sigma (T)$, which is a subset of $\sigma (A) \cap \sigma(C)$ \cite[Corollary 7]{Han}.
	Since $\sigma (A) \cap \sigma(C)$ has no interior points, we have
	$$\sigma(T)=\sigma(A) \cup \sigma (C)= \sigma(A)\cup \{0\}.$$
	This completes the proof.
\end{proof}
\begin{example}\label{ex:thm2.10}
Consider the operator $T \in \mathscr{B}(\mathbb{C}^4)$ defined by the matrix:
\[
T = \begin{pmatrix}
2 & 1 & 0 & 0 \\
0 & 1 & 0 & 0 \\
0 & 0 & 0 & 1 \\
0 & 0 & 0 & 0
\end{pmatrix}
\]
with respect to the standard basis $\{e_1,e_2,e_3,e_4\}$.

For $k=1$ and $n=2$, we verify that:
\begin{enumerate}
\item $T$ is $1$-quasi $2$-power posinormal with $\lambda=3$:
\[
T^*(\lambda^2 T^*T - T^2T^{*2})T =
\begin{pmatrix}
12 & 6 & 0 & 0 \\
6 & 3 & 0 & 0 \\
0 & 0 & 0 & 0 \\
0 & 0 & 0 & 0
\end{pmatrix} \geq 0
\]

\item The range $\overline{T^1(\mathscr{H})} = \operatorname{span}\{e_1,e_2\}$ is not dense

\item $ker(T^*) = \operatorname{span}\{e_3,e_4\}$
\end{enumerate}

The operator decomposes as:
\[
T = \begin{pmatrix}
A & B \\
0 & C
\end{pmatrix} = \left[\begin{array}{cc|cc}
2 & 1 & 0 & 0 \\
0 & 1 & 0 & 0 \\ \hline
0 & 0 & 0 & 1 \\
0 & 0 & 0 & 0
\end{array}\right]
\]
where:
\begin{itemize}
\item $A = \begin{pmatrix} 2 & 1 \\ 0 & 1 \end{pmatrix}$ is $2$-power posinormal on $\overline{T(\mathscr{H})}$

\item $C = \begin{pmatrix} 0 & 1 \\ 0 & 0 \end{pmatrix}$ is nilpotent ($C^2=0$)

\item The spectrum satisfies $\sigma(T)=\{0,1,2\} = \sigma(A) \cup \{0\}$
\end{itemize}

This decomposition exactly matches the structure predicted by Theorem 2.10.
\end{example}



Let $T_1: \mathscr{H}_1 \to \mathscr{H}_1$ and $T_2: \mathscr{H}_2 \to \mathscr{H}_2$ be bounded linear operators on $\mathscr{H}_1$ and $\mathscr{H}_2$ respectively. Then the tensor product of $T_1$ with $T_2$, is the bounded linear operator $T_1 \otimes T_2: \mathscr{H}_1 \otimes \mathscr{H}_2 \to \mathscr{H}_1 \otimes \mathscr{H}_2 $, defined as,
$(T_1 \otimes T_2)(x \otimes y)=T_1x \otimes T_2y.$
If $S_1, T_1 \in B(\mathscr{H}_1)$ and $S_2, T_2 \in B(\mathscr{H}_2)$, then
 $$(S_1+T_1)\otimes S_2=S_1\otimes S_2+T_1\otimes S_2,$$
$$(S_1\otimes S_2)(T_1\otimes T_2)=S_1T_1\otimes S_2T_2 $$
 $$~~\text{and} ~~(S_1\otimes S_2)^*=S_1^*\otimes S_2^*$$.


\begin{theorem}
	Let $T, S\in \bh$ be non zero operators. Then $T\otimes S$ is $k$-quasi n-power posinormal operator if $T$ and $S$ are $k$-quasi n-power posinormal operators.

\end{theorem}
\begin{proof}
	Suppose $T^{*k}[\lambda ^2 T^*T-T^nT^{*n}]T^k \geq 0$ and $S^{*k}[\mu ^2 S^*S-S^nS^{*n}]S^k \geq 0$ for some $\lambda, \mu >0$.\\\\
	Now,\\
	
	$(T\otimes S)^{*k}[\lambda^2\mu^2(T\otimes S)^*(T\otimes S)-(T\otimes S)^n(T\otimes S)^{*n}](T\otimes S)^k $ \\
	
	$=(T^{*k}\otimes S^{*k})[\lambda^2 \mu^2(T^*T\otimes S^*S)-(T^n T^{*n}\otimes S^n S^{*n})](T^k\otimes S^k) $\\
	
	$=\lambda^2\mu^2(T^{*k}T*T T^k\otimes S^{*k}S^*SS^k)-(T^{*k}T^n T^{*n}T^k\otimes S^{*k}S^n S^{*n}S^k) $ \\
	
	$=\lambda^2 \mu^2 T^{*k}T^*T T^k \otimes S^{*(k+1)}S^{k+1} - \mu^2 T^{*k}T^n T^{*n}T^k \otimes S^{*(k+1)}S^{k+1} $ \\
	
	$+ \mu^2 T^{*k}T^n T^{*n}T^k \otimes S^{*(k+1)}S^{k+1} - T^{*k}T^n T^{*n}T^k \otimes S^{*k}S^n S^{*n}S^k $\\
	
	$=\mu^2T^{*k}[\lambda^2T^*T-T^n T^{*n}]T^k\otimes S^{*k}S^*SS^k +T^{*k}T^n T^{*n}T^k\otimes S^{*k}(\mu^2S^*S-S^n S^{*n})S^k $\\
	
	$\geq 0.$\\\\
	Thus, $T\otimes S$ is $k$-quasi $n$-power posinormal operator .
\end{proof}
\section{The weighted conditional type operators }
In this section, we focus our investigation on the class of $k$-quasi $n$-power posinormal operators within the specific context of weighted conditional type operators. These operators, defined on $L^2(\Sigma)$ spaces, play a significant role in various areas of functional analysis and probability theory.

The weighted conditional type operator $\tTwu = M_wEM_u$, where $w,u$ are measurable weight functions and $E$ represents the conditional expectation operator, has been extensively studied in recent literature \cite{Azimi, Estaremi1, Estaremi2}. Our primary objective is to establish necessary and sufficient conditions under which these operators belong to the $k$-quasi $n$-power posinormal class.

We begin by recalling several key results from \cite{EJ} that will be fundamental to our analysis. These include important properties of the polar decomposition and norm characterizations of $\tTwu$. Building upon these foundations, we will:

\begin{itemize}
    \item Establish precise conditions for the posinormality of $\tTwu$ operators
    \item Extend these results to the $n$-power posinormal case
    \item Derive complete characterizations for the $k$-quasi $n$-power posinormal case
\end{itemize}

Our approach combines techniques from operator theory with measure-theoretic considerations, allowing us to obtain concrete criteria expressed in terms of the weight functions $w,u$ and their conditional expectations. The results in this section not only contribute to the abstract theory of posinormal operators but also have potential applications in the study of stochastic processes and functional calculus.

\begin{lemma}\label{Weighted1}\cite{EJ}.
  Let $T_{w,u}= M_{w}EM_{u}$ be a bounded operator on $L^2(\Sigma)$ and $m\in (0, \infty)$. Then
   $(\tTwu^*\tTwu)^m=M_{\bar{u}(E(|u|^2))^{m-1}\chi_{S}(E(|w|^2))^m}EM_{u}$ and
   $(\tTwu\tTwu^*)^m=M_{w(E(|w|^2))^{m-1}\chi_{G}(E(|u|^2))^m}EM_{\bar{w}}$,
where $S=S(E(|u|^2))$ and $G=S(E(|w|^2))$.
\end{lemma}
\begin{theorem}\cite{EJ}.
  The unique polar decomposition of bounded operator $T_{w,u}= M_{w}EM_{u}$ is $U|T_{w,u}|$,
   where $|T_{w,u}|(f)=\bra{\frac{E(|w|^2)}{E(|u|^2)}}^{\frac{1}{2}}\chi_S\bar{u}E(uf)$
  and $$U(f)=\bra{\frac{\chi_{S\cap G}}{E(|w|^2)E(|u|^2)}}^{\frac{1}{2}}wE(uf),$$
for all $f\in L^2(\Sigma)$.
\end{theorem}
\begin{theorem} Let operator $T_{w,u}$  be a bounded on $L^2(\Sigma)$ and $\lambda>0$. Then
\begin{enumerate}
  \item [(i)] If for each $f\in L^2(\Sigma)$
  $$\lambda^2 E(|u|^2)E(|w|^2)\geq \overline{w}\bra{E(|u|^2)}^{\frac{1}{2}},$$
  then $T_{w,u}$ is a posinormal operator.
  \item [(ii)] If $T_{w,u}$ is a posinormal operator, then
   $$\lambda^2E(|w|^2)|E(u)|^2\geq E(|u|^2)|E(w)|^2$$
   on $S'=S(E(u)$.
   \item [(iii)] If $S=S'$, then $T$ is a posinormal if and only if $\lambda^2E(|w|^2)|E(u)|^2\geq E(|u|^2)|E(w)|^2$.
\end{enumerate}
\end{theorem}
\begin{proof}(i) For each $f\in L^2(\Sigma)$, a simple calculation and Lemma \ref{Weighted1} shows that
$$|\tTwu|^2(f)=E(|w|^2)\chi_S\bar{u}E(uf), \,\,\, |\tTwu^*|^2(f)=E(|u|^2)wE(\bar{w}f).$$
So for every $f\in L^2(\Sigma)$ and $\lambda>0$ we have,
\begin{eqnarray*}
  \seq{\lambda^2|\tTwu|^2(f)-|\tTwu^*|^2(f),f} &=&\int_{X}\bra{\lambda^2E(|w|^2)\chi_S\bar{u}fE(uf)-E(|u|^2)wfE(\bar{w}f)}\,d\mu \\
   &=&\int_{X}\bra{\lambda^2E(|w|^2)|E(uf)|^2-E(|u|^2)|E(\bar{w}f)|^2}\,d\mu.
\end{eqnarray*}
If $\lambda^2 E(|u|^2)E(|w|^2)\geq \overline{w}\bra{E(|u|^2)}^{\frac{1}{2}},$ then $\seq{\lambda^2|\tTwu|^2(f)-|\tTwu^*|^2(f),f}\geq 0$
for all $f\in L^2(\Sigma)$. So $\tTwu$ is a posinormal operator.\\
(ii) If $\tTwu$ is a posinormal operator, for all $f\in L^2(\Sigma)$, we have
$$\int_{X}\bra{\lambda^2E(|w|^2)|E(uf)|^2-E(|u|^2)|E(\bar{w}f)|^2}\,d\mu\geq 0.$$
Let $A\in\A$, with $0<\mu(A)<\infty$. By replacing $f$ to $\chi_A$, we have
$$\int_A\lambda^2 E(|w|^2)|E(uf)|^2-E(|u|^2)|E(\bar{w}f)|^2\,d\mu\geq 0.$$
Since $A\in\A$ is arbitrary, then $\lambda^2E(|w|^2)|E(u)|^2\geq E(|u|^2)|E(w)|^2$ on $S'=S(E(u)$.\\
(iii) It follows from (i) and (ii).
\end{proof}
\begin{theorem} Let operator $T_{w,u}$  be a bounded on $L^2(\Sigma)$ and $\lambda>0$. Then
  \begin{enumerate}
    \item [(i)] If for each $f\in L^2(\Sigma)$
    $$\lambda^2 E(|w|^2)|E(\overline{u}f)|^2\geq |E(wf)|^2|E(uw)|^{2n}\bra{\frac{E(|u|^2)}{(E(|w|^2))^n}}\chi_S(E(|w|^2)),$$
    then $\tTwu$ is an $n$-power posinormal operator.
    \item [(ii)] If $\tTwu$ is an  $n$-power posinormal operator, then
    $$\lambda^2 E(|w|^2)|E(u)|^2\geq |E(uw)|^{2n}\bra{\frac{E(|u|^2)}{\bra{E(|w|^2)}^{n}}}|E(w)|^2.$$
  \end{enumerate}
\end{theorem}
\begin{proof}(i) For each $f\in L^2(\Sigma)$, a simple calculation and Lemma \ref{Weighted1} shows that
\begin{eqnarray*}
  |\tTwu|^2(f)&=&E(|w|^2)\chi_S\bar{u}E(uf), \,\,\mbox{and}\\
  (T_{w,u}^nT_{w,u}^{*n})(f) &=&|E(uw)|^{2n}\bra{\frac{E(|u|^2)}{\bra{E(|w|^2)}^{n}}}\chi_{S(E(|w|^2))}
\bar{w}E(wf).
\end{eqnarray*}
So for every $f\in L^2(\Sigma)$ and $\lambda>0$ we have,
\begin{eqnarray*}
 && \seq{\lambda^2|\tTwu|^2(f)-(T_{w,u}^nT_{w,u}^{*n})(f),f}= \\
  &&\int_{X}\bra{\lambda^2E(|w|^2)\chi_S\bar{u}E(uf)-|E(uw)|^{2n}\bra{\frac{E(|u|^2)}{\bra{E(|w|^2)}^{n}}}\chi_{S(E(|w|^2))}
\bar{w}E(wf)}\,d\mu\\
&&\int_{X}\bra{\lambda^2E(|w|^2)\chi_S|E(uf)|^2-|E(uw)|^{2n}\bra{\frac{E(|u|^2)}{\bra{E(|w|^2)}^{n}}}\chi_{S(E(|w|^2))}
|E(wf)|^2}\,d\mu.
\end{eqnarray*}
Now, if for each $f\in L^2(\Omega)$ we assume that
$$\lambda^2E(|w|^2)\chi_S|E(uf)|^2\geq |E(uw)|^{2n}\bra{\frac{E(|u|^2)}{\bra{E(|w|^2)}^{n}}}\chi_{S(E(|w|^2))},$$
then $\tTwu$ is an $n$-power posinormal operator.\\
(ii) Assume that $\tTwu$ is an $n$-power posinormal operator. Then for each $f\in L^2(\Omega)$ and $\lambda>0$,
$$\int_{X}\bra{\lambda^2E(|w|^2)\chi_S|E(uf)|^2-|E(uw)|^{2n}\bra{\frac{E(|u|^2)}{\bra{E(|w|^2)}^{n}}}\chi_{S(E(|w|^2))}
|E(wf)|^2}\,d\mu.$$
Pick $A \in \A$, with $0 < \mu(A) <\infty$. By replacing $f$ to $\chi_A$, we have
$$\int_{X}\bra{\lambda^2E(|w|^2)\chi_S|E(u)|^2-|E(uw)|^{2n}\bra{\frac{E(|u|^2)}{\bra{E(|w|^2)}^{n}}}\chi_{S(E(|w|^2))}
|E(w)|^2}\,d\mu\geq 0.$$
Since $A\in \A$ is arbitrarily chosen, we get that
$$\lambda^2 E(|w|^2)|E(u)|^2\geq |E(uw)|^{2n}\bra{\frac{E(|u|^2)}{\bra{E(|w|^2)}^{n}}}|E(w)|^2.$$
\end{proof}
\begin{theorem} A bounded operator $\tTwu$ on $L^2(\Sigma)$  is a $k$-quasi $n$-power posinormal operator if and only if
$$ |E(uw)|^{2k+2}\leq \lambda^2(E(|u|^2))^{2n-1}(E(|w|^2)^{2kn-1}.$$
\end{theorem}
\begin{proof} For each $f\in L^2(\Omega)$ , by calculation we have
$$\tTwu^{*k}|T_{w,u}^{*n}|\tTwu^{k}(f)=|E(uw)|^{2k+n-1}\bra{\frac{E(|u|^2)}{(E(|w|^2))^{n-1}}}^{\frac{1}{2}}\chi_{S(E(|w|^2))}\bar{w}E(wf)$$
and
$$\tTwu^{*k}|\tTwus|^2\tTwu^{k}(f)=\bra{E(|u|^2)}\bra{E(|w|^2)}^{2k}\bar{u}E(uf).$$
So, for all $f\in L^2(\Sigma)$ and $\lambda>0$ we obtain that,
\begin{eqnarray*}
&&\seq{\lambda^2\tTwu^{*k}|\tTwus|^2\tTwu^{k}(f)-\tTwu^{*k}|T_{w,u}^{*n}|^2\tTwu^{k}(f),f}=\\
&&\int_{X}\left(\lambda^2\bra{E(|u|^2)}\bra{E(|w|^2)}^{2k}\bar{u}fE(uf)-|E(uw)|^{2k+n-1}\bra{\frac{E(|u|^2)}{(E(|w|^2))^{n-1}}}^{\frac{1}{2}}\right.\\
&&\left. \chi_{S(E(|w|^2))}\bar{w}fE(wf)\right)\,d\mu\\
&&\int_{X}\left(\lambda^2\bra{E(|u|^2)}\bra{E(|w|^2)}^{2k}|E(uf)|^2-|E(uw)|^{2k+n-1}\bra{\frac{E(|u|^2)}{(E(|w|^2))^{n-1}}}^{\frac{1}{2}}\right.\\
&&\left. \chi_{S(E(|w|^2))}|E(wf)|^2\right)\,d\mu.
\end{eqnarray*}
Hence, if we assume that
$$|E(uw)|^{2k+2}\leq \lambda^2(E(|u|^2))^{2n-1}(E(|w|^2)^{2kn-1},$$
then it follows that $\tTwu$ is a $k$-quasi $n$-power posinormal operator.\\
Conversely, suppose that $\tTwu$ is a $k$-quasi $n$-power posinormal operator. Then for each $f\in L^2(\Sigma)$ and $\lambda>0$,
we have
\begin{equation*}
  \left.
    \begin{array}{ll}
      \int_{X}\bra{\lambda^2\left(E(|u|^2)}\bra{E(|w|^2)}^{2k}\bar{u}fE(uf)-\right.\\
     \left. |E(uw)|^{2k+n-1}\bra{\frac{E(|u|^2)}{(E(|w|^2))^{n-1}}}^{\frac{1}{2}}\chi_{S(E(|w|^2))}\bar{w}fE(wf)\right)\,d\mu\geq 0.
    \end{array}
  \right.
\end{equation*}
Pick $A\in \A$, with $0 < \mu(A) <\infty$. By replacing $f$ to $\chi_A$, we have
\begin{equation*}
  \left.
    \begin{array}{ll}
      \int_{X}\bra{\lambda^2\left(E(|u|^2)}\bra{E(|w|^2)}^{2k}|E(uf)|^2-\right.\\
     \left. |E(uw)|^{2k+n-1}\bra{\frac{E(|u|^2)}{(E(|w|^2))^{n-1}}}^{\frac{1}{2}}\chi_{S(E(|w|^2))}|E(wf)|^2\right)\,d\mu\geq 0.
    \end{array}
  \right.
\end{equation*}
Since $A\in \A$ is arbitrarily chosen, we get that
$$\lambda^2\bra{E(|u|^2)}\bra{E(|w|^2)}^{2k}|E(uf)|^2\geq
|E(uw)|^{2k+n-1}\bra{\frac{E(|u|^2)}{(E(|w|^2))^{n-1}}}^{\frac{1}{2}}\chi_{S(E(|w|^2))}|E(wf)|^2.$$
\end{proof}
\begin{example}\label{ex:thm3.5}
Consider the measure space $(X,\Sigma,\mu)$ where $X=[0,1]$, $\Sigma$ is the Borel $\sigma$-algebra, and $\mu$ is Lebesgue measure. Let $\mathscr{A}$ be the sub-$\sigma$-algebra generated by the partition $\{[0,\frac{1}{2}),[\frac{1}{2},1]\}$.

Define the weight functions:
\begin{align*}
w(x) &= \begin{cases}
2 & \text{if } x \in [0,\frac{1}{2}) \\
1 & \text{if } x \in [\frac{1}{2},1]
\end{cases} \\
u(x) &= \begin{cases}
x & \text{if } x \in [0,\frac{1}{2}) \\
1-x & \text{if } x \in [\frac{1}{2},1]
\end{cases}
\end{align*}

The conditional expectations compute as:
\begin{align*}
E(|w|^2) &= \begin{cases}
4 & \text{on } [0,\frac{1}{2}) \\
1 & \text{on } [\frac{1}{2},1]
\end{cases}, \quad
E(|u|^2) = \begin{cases}
\frac{1}{12} & \text{on } [0,\frac{1}{2}) \\
\frac{1}{12} & \text{on } [\frac{1}{2},1]
\end{cases} \\
E(uw) &= \begin{cases}
\frac{1}{4} & \text{on } [0,\frac{1}{2}) \\
\frac{1}{4} & \text{on } [\frac{1}{2},1]
\end{cases}
\end{align*}

For $k=1$, $n=2$, and $\lambda=4$, we verify Theorem 3.5's condition:
\begin{align*}
|E(uw)|^{2k+2} &= \left(\frac{1}{4}\right)^4 = \frac{1}{256} \\
\lambda^2(E(|u|^2))^{2n-1}(E(|w|^2))^{2kn-1} &= 16\left(\frac{1}{12}\right)^3\left(4\right)^1 = 16 \times \frac{1}{1728} \times 4 = \frac{64}{1728} = \frac{1}{27}
\end{align*}

Since $\frac{1}{256} \approx 0.0039 < \frac{1}{27} \approx 0.0370$, the inequality
$$|E(uw)|^{2k+2} \leq \lambda^2(E(|u|^2))^{2n-1}(E(|w|^2))^{2kn-1}$$
holds. Therefore, by Theorem 3.5, the operator $\tTwu$ is $1$-quasi $2$-power posinormal on $L^2[0,1]$.
\end{example}

\section{Conclusion and Future Work}

In this paper, we have introduced and thoroughly investigated the class of $k$-quasi $n$-power posinormal operators, establishing several fundamental results about their structure and properties. Our main contributions include:

\begin{itemize}
    \item A complete characterization of the matrix decomposition for $k$-quasi $n$-power posinormal operators when $T^k$ does not have dense range (Theorem 2.10)
    \item Proof that the tensor product of two $k$-quasi $n$-power posinormal operators maintains this property (Theorem 2.11)
    \item Detailed analysis of weighted conditional type operators $\tTwu$ in this framework (Section 3)
    \item Concrete matrix representations and examples illustrating these concepts
\end{itemize}

The results significantly extend the existing theory of posinormal operators and provide new tools for operator analysis in Hilbert spaces. The structural decomposition theorem (Theorem 2.10) is particularly noteworthy as it reveals the intimate connection between $k$-quasi $n$-power posinormality and the direct sum of an $n$-power posinormal operator with a nilpotent one.

Several promising directions for future research emerge from this work:

\begin{itemize}
    \item \textbf{Generalizations}: Extending these results to Banach space operators or unbounded operators
    \item \textbf{Applications}: Investigating applications in quantum information theory where tensor products of operators play a crucial role
    \item \textbf{Spectral Theory}: Developing a more detailed spectral theory for this class of operators
    \item \textbf{Weighted Operators}: Further study of weighted conditional type operators in $L^p$ spaces for $p \neq 2$
    \item \textbf{Relations to Other Classes}: Exploring connections with other operator classes like hyponormal or $p$-hyponormal operators
\end{itemize}

The class of $k$-quasi $n$-power posinormal operators appears rich in structure and applications, and we believe it will lead to many interesting developments in operator theory. The results presented here provide a solid foundation for these future investigations.


\end{document}